\documentclass[11pt,lettersize]{article}
\usepackage{latexsym,subfigure}
\usepackage{comment}
\usepackage{amsmath,amsthm,amsfonts,amssymb,graphicx,epsfig,latexsym,color}
\usepackage[notcite,notref]{showkeys}
\usepackage{epsf,enumerate}

\renewcommand{\Gamma}{G}
\title{Stable commutator length on mapping class groups}
\author{Mladen Bestvina, Ken Bromberg and Koji Fujiwara\thanks{The
    first two authors gratefully acknowledge the support by the National
    Science Foundation. The third author is supported in part by
Grant-in-Aid for Scientific Research (No. 23244005)}} 
\date{\today}

\newtheorem{thm}{Theorem}[section]
\newtheorem{lemma}[thm]{Lemma}
\newtheorem{cor}[thm]{Corollary}
\newtheorem{prop}[thm]{Proposition}
{}
\newtheorem*{A}{Theorem A}{}
\newtheorem*{B}{Theorem B}{}

\theoremstyle{remark}

\newtheorem{remark}[thm]{Remark}

\newtheorem*{definition*}{Definition}
\newtheorem*{remark*}{Remark}

\newcommand{\bY}{{\bf Y}}

\def\cC{{\mathcal C}}

\def\R{{\mathbb R}}

\def\Z{{\mathbb Z}}

\def\S{{\Sigma}}

\def\diam{\operatorname{diam}}

\def\scl{\operatorname{scl}}

\renewcommand{\>}{\rangle}

\begin{document}

\maketitle

\begin{abstract}
Let $\Gamma$ be a finite index subgroup of the mapping class group
$MCG(\S)$ of a closed orientable
surface $\S$, possibly with punctures.  We give a precise condition
(in terms of the Nielsen-Thurston decomposition) when an element
$g\in\Gamma$ has positive stable commutator length. In
addition, we show that in these situations the stable commutator
length, if nonzero, is uniformly bounded away from 0. The method works
for certain subgroups of infinite index as well and we show $scl$ is
uniformly positive on the nontrivial elements of the Torelli group.
The proofs use
our earlier construction in \cite{bbf} of group actions on
quasi-trees.
\end{abstract}

\section{Introduction and statement of results}
Let $G$ be a group, and $[G,G]$ its commutator subgroup. For an
element $g \in [G,G]$, let $cl(g)=cl_G(g)$ denote the {\it commutator length}
of $g$, the least number of commutators whose product is equal to $g$.
We define $cl(g)=\infty$ for an element $g$ not in $[G,G]$.  For $g
\in G$, the {\it stable commutator length}, $\scl(g)=\scl_G(g)$, is
defined by
$$\scl(g)=\liminf_{n \to \infty} \frac{cl(g^n)}{n} \le \infty.$$
See for example \cite{scl} for a thorough account on scl. 

A function $H:G \to \R$ is a {\it quasi-morphism}
if 
$$\Delta(H):=\sup_{x,y\in G}
  |H(xy)-H(x)-H(y)| < \infty.$$ $\Delta(H)$ is called the {\it defect} of $H$.
Recall that a quasi-morphism $H:G\to\R$ is {\it homogeneous} if
$H(g^m)=mH(g)$ for all $g\in G$ and $m\in \Z$.

A theme in the subject is to classify elements $g$ in a given group
for which $scl(g)>0$. Note that in the following situations $cl(g^n)$
is bounded and therefore $scl(g)=0$:
\begin{enumerate}[(a)]
\item $g$ has finite order,
\item $g$ is {\it achiral}, i.e. $g^k$ is conjugate to $g^{-k}$ for
  some $k\neq 0$,\footnote{of course, $scl(g)=0$ if $g^k$ is conjugate
    to $g^l$ for some $k\neq l$, but in mapping class groups this is
    possible only when $k=\pm l$ or $g$ has finite order}
\item (Endo-Kotschick \cite{EnKo}) $g=g_1g_2=g_2g_1$ and
  $g_1$ is conjugate to $g_2^{-1}$,
\item more generally, $g$ is expressed as a commuting
  product $$g=g_1\cdots g_p$$ and $g_i^{n_i}$ are all conjugate for
  some $n_i\neq 0$, and $$\sum_i \frac 1{n_i}=0,$$ 
\item $g=g_1\cdots g_p$ is a commuting product and $cl(g_i^n)=0$ are
  bounded for
  all $i$.
\end{enumerate}

Roughly speaking, our main theorem states that for mapping classes $g$
we have $scl(g)>0$ unless writing an arotational power of $g$ as a
commuting product as in the Nielsen-Thurston decomposition implies
$scl(g)=0$ by the above properties.

Brooks \cite{brooks} showed that in free groups
$scl(g)>0$ for every nontrivial element $g$. Then
Epstein-Fujiwara \cite{epstein-fujiwara}, generalizing the Brooks
construction, proved that in hyperbolic
groups the above obstructions (a) and (b) are the only ones, namely, if $g$ has
infinite order and $g$ is {\it chiral} (i.e. not achiral) then $scl(g)>0$.

For $G=MCG(\Sigma)$ questions due to Geoff Mess about commutator
length of powers of Dehn twists appear on Kirby's list \cite{kirby}.
Using 4-manifold invariants, Endo-Kotschick 
\cite{EnKo.inventiones} and Korkmaz \cite{korkmaz} prove that
$scl(g)>0$ if $g$ is a Dehn twist and Baykur \cite{baykur} gave an
argument based on the Milnor-Wood inequality that Dehn twist in the
boundary curve has $scl=\frac 12$. Endo-Kotschick \cite{EnKo} also
note that in $MCG(\Sigma)$ there are additional obstructions to
$scl>0$: if $g,h$ commute and $h$ is conjugate to $g$ then
$scl(gh^{-1})=0$; for example this occurs if $g,h$ are Dehn twists in
disjoint curves in the same orbit. By contrast,
Calegari-Fujiwara \cite{calegari-fujiwara} prove that if $g$ is
pseudo-Anosov and chiral 
then $scl(g)>0$. The argument is based on the action of $MCG(\Sigma)$
on the curve graph, which is hyperbolic by \cite{MM}, with
pseudo-Anosov classes acting as hyperbolic isometries.

In this paper, for a subgroup $G<MCG(\Sigma)$ of finite index of the
mapping class group of a closed orientable surface $\Sigma$ (possibly
with punctures) we characterize those elements $g\in G$ for which
$scl(g)>0$, or equivalently
there exists a quasi-morphism $G\to \R$ which is unbounded on the
powers of $g$. The new feature is that we consider actions of a
certain subgroup $\mathcal S<MCG(\Sigma)$ of finite index on
hyperbolic spaces constructed in \cite{bbf}. In this way, for {\it any}
nontrivial element of $\mathcal S$ there is an action where this
element is a hyperbolic isometry.

Here is our main result. The definition of {\it chiral
essential classes} will 
be given later, but we describe them informally below.

\begin{A}[Theorem \ref{chirality}]
Let $G<MCG(\Sigma)$ be a subgroup of finite index and $g\in G$. Then
$scl(g)>0$ if and only if some chiral equivalence class of pure
components of $g$ is essential.
\end{A}

Since $scl(g)>0$ if and only if $scl(g^k)>0$ for $k\neq 0$ we are free
to pass to powers, and we may assume that $g$ has the Nielsen-Thurston
decomposition in which there are no rotations, so $g$ is expressed as
a commuting product of Dehn twists and pseudo-Anosov pure components. 
Now consider the chirality of the pure components, e.g. every Dehn
twist is chiral.
Two chiral components of $g$ are {\it equivalent}
if they have nontrivial powers that are conjugate. A chiral
equivalence class is {\it essential} if after conjugating all to the
same supporting subsurface with the same pair of (un)stable
foliations, the product has infinite order.

We now state some corollaries and extensions of the Main Theorem.

\begin{B} Let $G$ be a subgroup of $MCG(\Sigma)$ of finite index.
\begin{itemize}
\item There is $\epsilon=\epsilon(G)>0$ so that if $g\in G$ with
  $scl_G(g)>0$ then $scl_G(g)>\epsilon$ (Proposition \ref{scl.bound}).
\item If $scl_G(g)=0$ then the sequence $cl_G(g^n)$ is bounded (Proposition \ref{cl.bounded}).
\item
There is a specific finite index subgroup $\mathcal S<MCG(\Sigma)$ so
that $scl_{\mathcal S}(g)>0$ for every nontrivial $g\in \mathcal S$ (the remark
after Theorem \ref{chirality}).
\item When $G=Ker[MCG(\Sigma)\to GL(H_1(\Sigma);\Z_3)]$, every
  exponentially growing element of $G$ has $scl_G>0$ (Corollary \ref{exponential}).
\item If $G=\mathcal T$ is the Torelli subgroup and $g\neq 1\in G$
  then $scl_G(g)>0$ (Theorem \ref{torelli}).
\end{itemize}
\end{B}

{\it Convention.} All our constants will depend linearly on the
previous constants, so for example when we say ``there is
$C=C(\delta,\xi)$'' we mean that $C$ is bounded by a fixed multiple of
$\delta+\xi+1$. All $(L,A)$-quasi-geodesics and
$(L,A)$-quasi-isometries will have $L$ uniformly bounded (e.g. by 2 or 4)
and $A$ will depend linearly on the previous constants.

{\it Acknowledgements:} We thank Danny Calegari for his interest and
to Dan Margalit for pointing us to the
reference \cite{BBM}.

\section{Review}

Here we review some background.

\subsection{scl}

The following facts will be used (see for example \cite{scl}). 

\begin{prop}\label{p:scl}
Let $H:G\to\R$ be a quasi-morphism and $\Delta(H)$ its defect.
\begin{enumerate}[(i)]
\item every quasi-morphism $H:G\to\R$ differs from a unique
  homogeneous quasi-morphism $\hat H$ by a bounded function; in
  fact
$$\hat H(g)=\lim_{n\to\infty}\frac{H(g^n)}n$$
\item Suppose $g\in [G,G]$. Then $\scl(g)\geq \frac{|\hat
  H(g)|}{4\Delta(H)} $, and if $H$ is homogeneous then $\scl(g)\geq
  \frac{|H(g)|}{2\Delta(H)}$.
\end{enumerate}
\end{prop}

Also note that any homogeneous quasi-morphism is constant on conjugacy
classes and it is a homomorphism when restricted to an abelian subgroup.

\begin{remark} 
The Bavard duality asserts that for a given $g$  
with $0<\scl(g) (<\infty)$, second inequality in (ii)
  is achieved for a suitable homogeneous $H$ with $0<\Delta(H)$. 
See \cite{bavard} and
  \cite{scl}. 
\end{remark}

\subsection{Quasi-trees and the bottleneck criterion}\label{s:quasi-trees}

All graphs will be connected and endowed with the path metric in which
edge lengths are 1.
A {\it quasi-tree} is a graph quasi-isometric to a tree. It is a
theorem of Manning \cite{pseudocharacters} that a graph $Q$ is a
quasi-tree if and only if it satisfies the {\it bottleneck property}:
there is a number $\Delta\geq 0$ (referred to as a {\it bottleneck
  constant}) such that any two vertices $x,y\in Q$
are connected by a path $\alpha$ with the property that any other path
connecting $x$ to $y$ necessarily contains $\alpha$ in its
$\Delta$-neighborhood. A quasi-tree is $\delta$-hyperbolic with
$\delta$ depending linearly only on $\Delta$.\footnote{i.e. it is
  bounded by a fixed multiple of $\Delta+1$}

\begin{thm}\cite{pseudocharacters}
Let $Q$ be a graph satisfying the bottleneck property with constant
$\Delta$. Then there is a tree $T$ and a $(4,A)$-quasi-isometry
$\alpha:Q\to T$ with $A$ depending linearly on $\Delta$.
\end{thm}

\begin{proof}
This is not exactly the way it is stated in \cite{pseudocharacters},
so we give an outline. We will assume that $\Delta\geq 1$ is an
integer; in general we replace it by the next larger integer and this
doesn't change the conclusion. The idea is to fix $R=20\Delta$ and a
base vertex $*\in Q$, and then define the vertices of $T$ as the path
components of $B(*,R),B(*,2R)-B(*,R),B(*,3R)-B(*,2R),\cdots$. Manning
produces an explicit quasi-isometry $\beta:T\to Q$ satisfying
$$8\Delta d(x,y)-16\Delta\leq d(\beta(x),\beta(y))\leq 26\Delta
d(x,y)$$
and which is $20\Delta$-almost onto. If the metric on $T$ is rescaled
by the factor $8\Delta$, $\beta$ becomes a
$(4,16\Delta)$-quasi-isometric embedding and the standard inverse 
is a $(4,A)$-quasi-isometry for $A$ a linear function of
$\Delta$. 
\end{proof}

In particular, if $[a,b]$ is a
$(2,10\delta+10)$-quasi-geodesic in $Q$, the image in $T$ is contained in
the $\epsilon$-neighborhood of the segment spanned by the images of
$a$ and $b$, and $\epsilon$ is a linear function of $\Delta$.

\subsection{Quasi-axes}

Unfortunately, in general, hyperbolic isometries do not have axes. The
notion of quasi-axes (with uniform constants) will serve as a
surrogate.  Assume $g:X\to X$ is a hyperbolic isometry of a
$\delta$-hyperbolic graph.  Let $x_0\in X$ be a vertex with
$D=d(x_0,g(x_0))$ minimal possible. This implies that
$d(x_0,g^2(x_0))\geq 2D-4\delta-4$ for otherwise a vertex near the
middle of $[x_0,g(x_0)]$ would be displaced less than $D$. In other
words, the piecewise geodesic $[x_0,g(x_0)]\cup [g(x_0),g^2(x_0)]$ is
a $(1,4\delta+4)$-quasi-geodesic of length $2D$. Now recall that local
quasi-geodesics are global quasi-geodesics: a $(100\delta+100)$-local
$(1,4\delta+4)$-quasi-geodesic is a global
$(2,10\delta+10)$-quasi-geodesic (see e.g. \cite[Theorem
  III.1.13]{bridson-haefliger} and \cite[Theorem
  4.1]{CDP}\footnote{\cite{bridson-haefliger} proves this for local
  geodesics and \cite{CDP} proves that local $(L,A)$-quasi-geodesics
  are $(L',A')$-quasi-geodesics; the arguments prove our
  assertion}). We can arrange that $\cup_i [g^i(x_0),g^{i+1}(x_0)]$ is
$g$-invariant, and we refer to it as a {\it quasi-axis} of $g$; it is
a $(2,10\delta+10)$-quasi-geodesic. Any two quasi-axes of $g$ are in
uniformly bounded Hausdorff neighborhoods of each other, and the bound
is a linear function of $\delta$. A {\it virtual quasi-axis} of $g$ is
a quasi-axis of a power of $g$. Thus every hyperbolic isometry has a
virtual quasi-axis.

It is convenient to introduce the following notation: for hyperbolic
isometries $g,h$ of a $\delta$-hyperbolic space $X$ we
write $$\Pi_g(h)=\Pi_g^X(h)\leq\eta$$ if the (nearest point) projection of any
virtual quasi-axis of $h$ to any virtual quasi-axis of $g$ has
diameter $\leq\eta$. Note that for any two pairs of choices of
virtual quasi-axes the diameters of projections differ by a number
bounded by a linear function of $\delta$.

We will also write $$\tilde\Pi_g(h)=\tilde\Pi_g^X(h)\leq\eta$$ if
$\Pi_g(h')\leq\eta$ for every conjugate $h'$ of $h$.

\subsection{WWPD}

Throughout the paper we will be concerned with the following
setting:
\begin{itemize}
\item
$G$ is a group acting on a $\delta$-hyperbolic graph $X$, 
\item $g\in G$ is a hyperbolic element,
\item $C=C(g)<G$ is a subgroup that fixes the points $g^{\pm\infty}$ at
  infinity fixed by $g$; equivalently, for every virtual quasi-axis
  $\ell$ the orbit $C\ell$ is
  contained in a Hausdorff neighborhood of $\ell$ and no
  element of $C$ flips the ends, and 
\item there is $\xi=\xi_g>0$ such that for every $\gamma\in G-C$ we
  have $\Pi_g(g^\gamma)\leq\xi$.
\end{itemize}

When these properties hold we will say that $(G,X,g,C)$ satisfy
$WWPD$. This is a weakening of $WPD$ \cite{bf:gt} which requires in
addition that $C$ be virtually cyclic\footnote{Strictly speaking $WPD$
  allows flips; one could talk about $WWPD^{\pm}$ vs. $WWPD$ but
  we will keep it simple.}. As we shall see below, every mapping
class group has a torsion-free subgroup $\mathcal S$ of finite index
such that {\it every} nontrivial element $g$ admits a WWPD action.
For
example, when $g$ is a pseudo-Anosov mapping class, we may take the
curve graph for $X$ (and then the action is $WPD$), but for a general
element we will use the construction in \cite{bbf} and obtain only
WWPD actions.

It is tempting to think that in this situation there is always a
homomorphism $C\to\R$ given by the (signed) translation length and
then modify $X$ so that $C$ acts by translations on a fixed line (axis
of $g$). However, in general this is not possible since $C\to\R$ is
only a quasi-morphism. This leads to the perhaps surprising fact that
(many) mapping class groups can act by isometries on a
$\delta$-hyperbolic graph with a Dehn twist having positive
translation length, while in general this is impossible on complete
$CAT(0)$ spaces, as observed by 
Martin Bridson \cite{bridson}.

\subsection{Mapping class group and the curve graph}

Let $\S$ be a closed orientable surface, possibly with punctures. The
mapping class group $MCG(\S)$ of $\S$ is the group of orientation
preserving homeomorphisms preserving the set of punctures, modulo
isotopy rel punctures. The curve graph $\mathcal C(\S)$ has a vertex
for every isotopy class of essential simple closed curves in $\S$, and
an edge corresponding to pairs of simple closed curves that intersect
minimally.

It is a fundamental theorem of Masur and Minsky \cite{MM} that the
curve graph is hyperbolic. Moreover, they show that an element $g$
acts hyperbolically if and only if $g$ is pseudo-Anosov, and that the
translation length 
$$\tau_g=\lim\frac{d(x_0,g^n(x_0))}n$$
of $g$ is uniformly bounded below by a positive
constant that depends only on $\S$.

Recall also that to an annulus $A$ one associates the ``curve
graph'' quasi-isometric to a line, whose vertices are represented by
isotopy classes of spanning arcs and edges by disjointness.

It was proved in \cite{bf:gt} that the action of $MCG(\S)$ on the curve
graph $\mathcal C(\S)$ satisfies $WPD$. This means that every
hyperbolic element (i.e. a pseudo-Anosov class) $g$ is a $WPD$ element,
that is, for every $g$ 
there exists $\xi_g$ such that $\Pi_g(g')\leq\xi_g$ for every
conjugate $g'$ of $g$ whose virtual quasi-axes aren't parallel to
those of $g$ and the stabilizer of $g^{\pm\infty}$ is virtually cyclic. 
Bowditch \cite{bhb:tight} improved
this result and showed that the action of $MCG(\S)$ on $\mathcal C(\S)$
is {\it acylindrical}. Denoting by $\delta$ the hyperbolicity constant
of $\mathcal C(\S)$ this means that if $x,y$ are two vertices
sufficiently far apart, say the distance at least $M$, then the set
$$\{h\in MCG(\S)\mid d(x,h(x))\leq 10\delta, d(y,h(y)\leq 10\delta\}$$
is finite and has cardinality bounded by some $N=N(\S)<\infty$. This
allows us to estimate $\xi_g$ as follows.

\begin{lemma}\label{brian}
There are constants $A,B$ that depend only on the surface $\S$ 
so that we may take
$\xi_g=A+B\tau_g$, where $\tau_g$ is the translation length of
$g$. More generally, if $f$ is another hyperbolic element with $\tau_f\leq
\tau_g$ then $\Pi_g(f)\leq\xi_g$
or else $f$ and $g$ have parallel virtual quasi-axes.
\end{lemma}

\begin{proof}
First note that after replacing $g$ with a bounded power we may assume
that $g$ has a quasi-axis; this is
because the translation length is bounded below. 
We will prove the lemma assuming $g$ has an axis $\ell$; the case of a
quasi-axis requires straightforward changes. If the projection of
$h(\ell)$ to $\ell$ has diameter $D>10\delta+10$, then there are
segments $I\subset\ell$ and $J\subset h(\ell)$ of length $D-4\delta$
that are in each other's $2\delta$-neighborhood. Denote by
$\phi=hgh^{-1}$, the conjugate of $g$ with axis $h(\ell)$. We will
assume that $I$ and $J$ are oriented in the same direction; otherwise
replace $\phi$ by $\phi^{-1}$. If $D>N\cdot\tau(g)+M+10\delta+10$,
the elements $1,g\phi^{-1},g^2\phi^{-2},\cdots,g^{N}\phi^{-N}$ move
each point of a segment of length $M$ a distance $\leq 10\delta$, so
from acylindricity we deduce $g^i\phi^{-i}=g^j\phi^{-j}$ for some
$i<j$, i.e. $g^{i-j}=\phi^{i-j}$, so in particular $\ell$ and
$h(\ell)$ are parallel.

For the second part, we increase $B$ by 1 and assume that $f$ violates
the conclusion. Then $f(\ell)$ has
projection to $\ell$ of diameter $>\xi_g$, so we must have
that $f(\ell)$ is parallel to $\ell$ and the conclusion follows.  
\end{proof}

The following lemma was proved in \cite{bbf} in the case of closed
surfaces; when $\Sigma$ has punctures the statement is easily reduced
to the closed case by doubling.
(In \cite{bbf} we find a subgroup which acts trivially in $\Z/2$-homology,
but we can further take a finite index subgroup which acts trivially 
in $\Z/3$-homology.)

\begin{prop}\cite[Lemma 4.7]{bbf}\label{BBF}
  There is a finite index normal subgroup $\mathcal
  S\subset MCG(\Sigma)$ which is torsion-free, fixes all punctures,
  acts trivially in $\Z/3$-homology of $\S$, and for every $h\in \mathcal S$
  and every simple closed curve $\alpha$ on $\Sigma$,
  $i(\alpha,h(\alpha))=0$ implies $h(\alpha)=\alpha$.
\end{prop}

When $S\subset\Sigma$ is a $\pi_1$-injective subsurface we denote by
$\hat S$ the surface obtained from $S$ by collapsing each boundary
component to a puncture. Note that every mapping class $f:\S\to\S$
that preserves $S$ induces a mapping class $\hat f:\hat S\to\hat S$.
The following is immediate from a theorem of Ivanov \cite{ivanov}
(see Section \ref{section.example}).

\begin{cor}\label{involution}
Let $f\in \mathcal S$ preserve a subsurface $S\subset\Sigma$. 
If $\hat f$ has finite order in $MCG(\hat S)$ then
$\hat f=id$.
\end{cor}

\subsection{The projection complex}

Recall \cite{mm2} that when $S,S'$ are $\pi_1$-injective subsurfaces
of $\S$ with $\partial S\cap\partial S'\neq\emptyset$ there is a
coarse {\it subsurface projection} $\pi_S(S')\subset\mathcal C(S)$, a
uniformly bounded subset of the curve complex $\mathcal C(S)$ obtained
by closing up each component of $\partial S'\cap S$ along $\partial
S$. Let $\bY$ be an $\mathcal S$-orbit of subsurfaces of $\S$, where
$\mathcal S\subset MCG(\S)$ is the subgroup as in Proposition
\ref{BBF}. Then distinct subsurfaces in $\bY$ have intersecting
boundaries (see \cite{bbf}) and the following two properties below
hold (for the first see \cite{behrstock}, and for a simple proof due
to Leininger see \cite{mangahas,mangahas2}). When $A,B,C\in\bY$ define
$d_A^\pi(B,C)=\diam\{\pi_A(B)\cup\pi_A(C)\}$. Then there is $\eta>0$
such that
\begin{itemize}
\item of the three numbers $d^\pi_A(B,C),d^\pi_B(A,C),d^\pi_C(A,B)$ at
  most one is larger than $\eta$, and
\item for every $A,B\in\bY$ the set $\{C\in\bY\mid
  d^\pi_C(A,B)>\eta\}$ is finite.
\end{itemize}

Section 3 of \cite{bbf} proves the following theorem.

\begin{prop}\label{bY}
Let $\bY$ be an $\mathcal S$-orbit of subsurfaces of $\Sigma$. Then
$\mathcal S$ acts by isometries on a hyperbolic graph $\cC(\bY)$ with
the following properties:
\begin{enumerate}[(i)]
\item For every surface $S\in\bY$ the curve graph $\mathcal C(S)$ is
  embedded isometrically as a convex subgraph in $\cC(\bY)$, and when $S\neq
  S'$ then $\mathcal C(S)$ and $\mathcal C(S')$ are disjoint.
\item The inclusion $$\bigsqcup_{S\in\bY} \mathcal C(S)\hookrightarrow
  \cC(\bY)$$ is $\mathcal S$-equivariant, where on the left $\phi\in\mathcal S$
  sends a curve $\alpha\in \mathcal C(S)$ to the curve
  $\phi(\alpha)\in \mathcal C(\phi(S))$.
\item For $S\neq S'$ the nearest point projection to $\mathcal C(S')$
  sends $\mathcal C(S)$ to a uniformly bounded set, and this set is
  within uniformly bounded distance from $\pi_{S'}(S)$.
\item Assume $g\in \mathcal S$ is pure, i.e. supported on $S\in\bY$
  and the restriction is pseudo-Anosov or, in case $S$ is an annulus,
  a power of a Dehn twist. Denote by $C$ the subgroup of $\mathcal S$
  consisting of elements $f$ that leave $S$ invariant and, if $S$ is
  not an annulus, $\hat f:\hat S\to\hat S$ preserves the stable and
  unstable foliations of $\hat g$. Then $(\mathcal S,\cC(\bY),g,C)$ satisfies
  $WWPD$.
\end{enumerate}
\end{prop}

\begin{proof}
The graph $\cC(\bY)$ is constructed in Section 3.1 of \cite{bbf} from
which it is clear that (ii) holds.
Hyperbolicity is Theorem 3.15, convexity is Lemma 3.1 and (iii)
is Lemma 3.11.
For (iv) use Corollary \ref{involution} to see that there
are no flips. If $g'$ is conjugate to $g$ and its virtual quasi-axis
is contained in $\mathcal C(S)$, then the projection to a quasi-axis $\ell$
of $g$ is uniformly bounded (or the two are parallel) by the acylindricity
of $\mathcal C(S)$. If the virtual quasi-axis of $g'$ is contained in
some other $\mathcal C(S')$ the projection to $\ell$ is bounded by the
third bullet.
\end{proof}

\begin{lemma}\label{bounded.projection}
Let $\phi\in\mathcal S$ be supported on a subsurface
$F$ so that $\phi|F$ is pseudo-Anosov (or a Dehn twist if $F$ is an
annulus). Also suppose that $F$ does not contain any $S\in\bY$. Then 
for each $S \in \bY$, there exists
a vertex in $\cC(\bY)$ such that the 
nearest point projection to $\cC(S)$ of its $\phi$-orbit is uniformly
bounded (independently of $F$, $\phi$ and $S$; the bound depends only
on $\Sigma$).

Thus, if $\phi$ is hyperbolic in $\cC(\bY)$, its virtual quasi-axis
can intersect $\cC(S)$
only in a uniformly bounded length segment.
\end{lemma}

\begin{proof}
If $S\cap F=\emptyset$ then 
$\phi$ fixes $\cC(S)$ pointwise ($\phi$ is elliptic) and the claim is clear.
Otherwise $\partial F\cap S\neq\emptyset$ and all $\phi^i(\partial
F)$, $i\in\Z$,
have the same projection to $S$ in $\Sigma$. 
The first part of the lemma follows from
Proposition \ref{bY}(iii).

Now, by an elementary argument in $\delta$-hyperbolic geometry,
if $\phi$ is hyperbolic and has a virtual quasi-axis, then the $\phi$-orbit
of a point on it has the smallest projection (in diameter) to $\cC(S)$  
among the $\phi$-orbits of points (up to a constant depending on $\delta$).
Therefore the last part of the claim follows. 
\end{proof}

\subsection{Promoting hyperbolic spaces to quasi-trees}

In this section we promote a $WWPD$ action $(G,X,g,C)$ with $X$ a
$\delta$-hyperbolic graph to a $WWPD$
action $(G,Q,g,C)$ where $Q$ is a quasi-tree.

\begin{prop}\label{promotion}
Let $X$ be a $\delta$-hyperbolic graph and assume $(G,X,g,C)$
satisfies $WWPD$ with the constant $\xi=\xi_g^X$. Then there is an
action of $G$ on a quasi-tree $Q$ such that:
\begin{enumerate}[(i)]
\item The bottleneck constant $\Delta=\Delta(\delta,\xi)$ for $Q$
  depends only on $\delta$ and $\xi=\xi^X_g$, and it is bounded by a
  multiple of $\delta+\xi+1$,
\item $(G,Q,g,C)$ satisfies $WWPD$ with $\xi_g^Q$ bounded by a multiple
  of $\delta+\xi+1$, 
\item if $h\in G$ is elliptic on $X$ then $h$ is elliptic on $Q$,
\item if $h\in G$ is hyperbolic on $X$ and if
  $\tilde\Pi^X_g(h)\leq\eta$ then either $h$ is elliptic on $Q$, or $h$
  is hyperbolic on $Q$ and $\tilde\Pi^Q_g(h)\leq\eta+P$ for
  some constant $P=P(\delta,\xi)$ which is a fixed multiple of
  $\delta+\xi+1$. 
\end{enumerate}
\end{prop}

\begin{proof} This is also a special case of the construction in
  \cite{bbf}.  Consider the conjugates of $g$, and say two are
  equivalent if they have parallel quasi-axes. For each equivalence
  class take the union of all quasi-axes of all of its members with
  the subspace metric -- this is a quasi-line. The collection $\bY$ of
  all these quasi-lines satisfies the axioms in Section 3 of
  \cite{bbf} since by assumption the projections are uniformly
  bounded. The space $\cC(\bY)$ constructed there is a quasi-tree by
  Theorem 3.10 of \cite{bbf}, and we name it $Q$. 
$\cC(\bY)$ contains the quasi-line for each $Y$.
The main observation
  for the proof of (i) is that the constant $K$ used in the definition
  of the projection complex depends only on $\delta$ and $\xi$ and the
  dependence is linear. Then (ii) follows from Lemma 3.11 of
  \cite{bbf} and (iii) is clear from the construction.

  Suppose $h$ is hyperbolic
  in $Q$ and has long overlap with $\mathcal C(Y)$, where $Y$ is one of the
  quasi-lines. Then $h^{-N}(Y)$ and $h^N(Y)$ have large projection in $Y$
  measured in $Q$, hence also in $X$ (this again uses Lemma 3.11 in
  \cite{bbf}). But then a virtual quasi-axis of $h$ has large
  projection to $Y$ in $X$.
\end{proof}

\section{Construction of quasi-morphisms}

In this section we show how to construct quasi-morphisms $G\to\R$ if
$(G,Q,g,C)$ satisfies $WWPD$ and $Q$ is a quasi-tree, generalizing the
Brooks construction.

\begin{prop}\label{bf++}
For every $\Delta$ there is $M=M(\Delta)$, a fixed multiple of
$\Delta+1$, such that the following
holds. Let $(G,Q,g,C)$ satisfy $WWPD$
where $Q$ is a quasi-tree with bottleneck constant $\Delta$ and assume
$\tau_g\geq
\xi_g+M$. Then there is a quasi-morphism $F:G\to\R$ such that
\begin{enumerate}[(a)]
\item the defect of $F$ is $\leq 12$,
\item $F$ is unbounded on the powers of $g$; more precisely, 
$$\hat F(g)\geq \frac 12$$
where $\hat F$ is the homogeneous quasi-morphism equivalent to $F$,
and moreover if $h$ is hyperbolic with virtual quasi-axes parallel to
those of $g$ and both $g,h$ translating in the same direction then 
$$\frac {\hat F(h)}{\tau_h}=\frac {\hat F(g)}{\tau_g}$$
and in particular $\hat F(h)\geq \frac{\tau_h}{2\tau_g}$,
\item $F$ is bounded on the powers of any elliptic element of $G$, and
\item $F$ is bounded on
  the powers of any hyperbolic element $\alpha$ such that 
$\tilde\Pi_g(\alpha)\leq\tau_g-M$.
\end{enumerate}
\end{prop}

\begin{proof}
The proof is a modification of the classical Brooks construction for
free groups \cite{brooks}. There are two variants, one counts the
number of {\it all} subwords of a given word isomorphic to a fixed
word $w$, and the other counts the maximal number of {\it
  non-overlapping} subwords isomorphic to $w$. The first version is
more convenient when working with coefficients (we worked this out in
the $WPD$ setting in \cite{bbf2}). The second version is more
convenient when control on the defect is important, and this is the
version we pursue here.

We start by fixing an $(4,A)$-quasi-isometry $\phi:Q\to T$ to a tree
$T$ and a constant $\epsilon\geq 0$ so that the $\phi$-image of a
$(2,10\delta+10)$-quasi-geodesic $[a,b]$ is in the
$\epsilon$-neighborhood of $[\phi(a),\phi(b)]$. Note that $A,\epsilon$
depend only on $\Delta$ and can be arranged to be fixed multiples
of $\Delta+1$, see Secton \ref{s:quasi-trees}.  Let $x_0$ be a vertex
with $D=d(x_0,g(x_0))$ minimal possible and let $w=[x_0,g(x_0)]$,
viewed as an oriented segment. By taking $M$ sufficiently large we may
assume that $D>>\delta,A,\epsilon$ and the union of $\<g\>$-translates of
$w$ forms a quasi-axis $\ell$ of $g$.  From this data we will
construct a quasi-morphism $F=F_{\phi,\epsilon,w}:G\to\R$.

A {\it copy} of $w$ is a translate $\gamma w$, also viewed as an oriented
segment. For $q,q'\in Q$ we write $\gamma w\overset\circ\subset
[q,q']$ if there exists $\beta\in G$ such that $\phi(\beta\gamma w)$
is contained in the $\epsilon$-neighborhood of the segment
$[\phi(\beta(q)),\phi(\beta(q'))]$ and $\phi(\beta(q))$ is closer to
the $\phi$-image of the initial endpoint of $\beta\gamma w$ than
the terminal endpoint. Since $w$ is long compared
to $\epsilon$ and the quasi-isometry constants of $\phi$, the
condition says that the copy $\gamma w$ is nearly contained in
$[q,q']$ in the oriented sense. Note that the notion is equivariant,
i.e. if $\gamma w\overset\circ\subset [q,q']$ then $\beta\gamma
w\overset\circ\subset [\beta(q),\beta(q')]$ for any $\beta\in
G$. Also, if $\gamma w\overset\circ\subset [q,q']$ and $\beta'\in G$ is
arbitrary, then $\phi(\beta'\gamma w)$
is contained in the $\epsilon'$-neighborhood of the segment
$[\phi(\beta'(q)),\phi(\beta'(q'))]$, where $\epsilon'$ also depends
linearly on $\delta$.

We say that two copies $\gamma w$ and $\gamma' w$ are {\it
  non-overlapping} if for some $\beta\in G$ the images
$\phi(\beta\gamma w)$ and $\phi(\beta\gamma' w)$ are disjoint. This
notion is also equivariant, and if $\beta'\in G$ is arbitrary the
intersection $\phi(\beta'\gamma w)\cap\phi(\beta'\gamma' w)$ has
uniformly bounded diameter, say by $\epsilon''$, which also depends
linearly on $\delta$. The constant $M$ and hence the length of $w$ will
be large compared to all these constants.

Now define the {\it non-overlapping count}
$N_w(q,q')$ as the maximal number of pairwise non-overlapping copies
$\gamma w\overset\circ\subset [q,q']$. To see that this number is
finite, note that the projection to $T$ of any $\gamma
w\overset\circ\subset [q,q']$ has a long overlap with
$[\phi(q),\phi(q')]$ while the pairwise overlaps are bounded.

{\it Claim.} Let $\delta$ be the hyperbolicity constant for $Q$ and
assume $|w|>>\delta$. If $r\in Q$ is
$2\delta$-close to a geodesic from $q$ to $q'$ then
$$|N_w(q,q')-N_w(q,r)-N_w(r,q')|\leq 2$$

Indeed, the union of maximal collections for $(q,r)$ and $(r,q')$
gives a non-overlapping collection for $(q,q')$, perhaps after
removing the two copies closest to $r$, and conversely, and maximal
collection for $(q,q')$ breaks up into two non-overlapping collections
for $(q,r)$ and $(r,q')$, perhaps after removing two copies closest to
$r$.

Now define $F:G\to\R$ by 
$$F(\alpha)=N_w(x_0,\alpha(x_0))-N_{w^{-1}}(x_0,\alpha(x_0))=
N_w(x_0,\alpha(x_0))-N_w(\alpha(x_0),x_0)$$
It is straightforward to
check $(a)-(d)$.

{\it Proof of (a).}
This is the standard Brooks tripod argument.
Let $\alpha,\beta\in G$ and let $r\in Q$ be within $2\delta$ of each
of 3 geodesics joining $x_0,\alpha(x_0),\beta\alpha(x_0)$. 

Now we have $N_w(x_0,\alpha(x_0))\sim N_w(x_0,r)+N_w(r,\alpha(x_0))$ by the
Claim. Write 5 more such equalities, for each oriented side of the
triangle $x_0,\alpha(x_0),\beta\alpha(x_0)$ and note that e.g.
$N_w(\alpha(x_0),\alpha\beta(x_0))=N_w(x_0,\beta(x_0))$.
Adding these (approximate) equalities, we find that
$$|F(\alpha\beta)-F(\alpha)-F(\beta)|\leq 12$$

{\it Proof of (b).}
Note that $g^{2k}(w)$ are non-overlapping, $k\in\Z$.
Thus
$N_w(x_0,g^{2k}(x_0))\geq k$ for every $k=1,2,\cdots$. It remains to
observe that $N_w(g^{k}(x_0),x_0)=0$ by the $WWPD$ assumption. Thus
$F(g^{2k})\geq k$ and so $\hat F(g)\geq \frac 12$ where $\hat F$ is
the homogenous quasi-morphism equivalent to $F$.

{\it Proof of (c).}
If the orbit $\{\alpha^i(x_0)\}$ is bounded, then so are the
translates $\{\beta\alpha^i(x_0)\}$ and their $\phi$-images, and hence
$N_{w^{\pm 1}}(x_0,\alpha^i(x_0))$ are uniformly bounded, and so are
$F(\alpha^i)$. 

{\it Proof of (d).} If $F(\alpha^N)\neq 0$ for large $N$, then there
must be a copy of $w$ near a virtual quasi-axis of $\alpha$ and we
see that the projection of this quasi-axis to the corresponding copy
of $\ell$ will contain the copy of $w$ except for segments
near the endpoints bounded by a fixed multiple of $\Delta+1$.
\end{proof} 

\begin{cor}
\label{bf+++}
There is a constant $R=R(\delta,\xi)$, a fixed multiple of
$\delta+\xi+1$, so that the following holds.
Let $(G,X,g,C)$ satisfy $WWPD$ where $X$ is a $\delta$-hyperbolic
graph and $\tau_g\geq R$.
Then there is a
quasi-morphism $F:G\to\R$ such that
\begin{enumerate}[(a)]
\item the defect of $F$ is $\leq 12$,
\item $F$ is unbounded on the powers of $g$; in fact $\hat F(g)\geq
  \frac 12$, 
\item $F$ is bounded on the powers of any elliptic element of $G$, and
\item $F$ is bounded on the powers of any hyperbolic element $\alpha$ 
with $\tilde\Pi_g(\alpha)\leq\tau_g-R$.
\end{enumerate}
\end{cor}

\begin{proof}
This is immediate from Propositions \ref{promotion} and \ref{bf++}.
\end{proof} 

\begin{remark}\label{bf++++}
In applications we will not necessarily have $\tau_g\geq R$, but will
have to pass to a power $g^N$ of $g$ to achieve this. For our uniform
estimates it will be important that $N$ is uniformly bounded. 
In the setting of the curve graph $\mathcal C(\S)$ of a fixed surface
$\S$ this follows from two facts:
\begin{itemize}
\item $\tau_g\geq\epsilon_\S>0$ for every hyperbolic $g$
  \cite[Proposition 4.6]{MM}, and
\item $\xi_g$ is bounded by a fixed multiple of $\tau_g+1$
  (see Lemma \ref{brian}).
\end{itemize}

Similarly, uniformity on powers holds in hyperbolic spaces $\mathcal
C(\bY)$ constructed in Proposition \ref{bY}. More precisely, if $g$ is
supported on a subsurface $S\in\bY$ and the restriction is
pseudo-Anosov, then its translation length in $\bY$ is equal to its
translation length in $\cC(S)$ (this follows from Proposition
\ref{bY}(i) and (ii)) and the projections of quasi-axes of conjugates
are bounded by a fixed multiple of $\tau_g+1$ (this follows from
Proposition \ref{bY}(iii) and Lemma \ref{brian}).
\end{remark}

\section{Stable commutator length on mapping class
  groups}\label{scl}

Now assume that $G<MCG(\S)$ is a finite index subgroup and $g\in G$. By the
Nielsen-Thurston theory (see e.g. \cite{casson-bleiler}) there is a
unique minimal $g$-invariant collection $\mathcal C$ (possibly empty) of
pairwise disjoint simple closed curves which are non-parallel and no
curve bounds a disk or a punctured disk and so that after replacing $g$ by a
power:
\begin{itemize}
\item each puncture of $\S$ is fixed,
\item
each curve in $\mathcal C$ is $g$-invariant, 
\item each
component of $\S-\cup_{c \in \mathcal C}\mathcal C$ is $g$-invariant,
\item the restriction of $g$ to each complementary component is
  homotopic to
  identity or a pseudo-Anosov homeomorphism.
\end{itemize}

Let $S_i$ be a complementary component on which $g$ is
pseudo-Anosov. Collapsing all boundary components to punctures
produces a closed surface $\hat S_i$ with punctures and $g$ induces a
pseudo-Anosov homeomorphism $\hat g_i:\hat S_i\to\hat S_i$. There is a
(projectively) $\hat g_i$-invariant measured (singular) foliation
$\hat{\mathcal F_i}$ on $\hat S_i$ without saddle connections. Each
puncture is a $k$-prong singularity for some $k=1,2,\cdots$ (when
$k=2$ it is a regular point). After passing to a higher power of $g$
we may assume that 
\begin{itemize}
\item 
each such $\hat g_i$ is {\it arotational} i.e. all
prongs (directions of leaves) out of any puncture are fixed. 
\end{itemize}
We may reverse the collapsing process and blow up the newly created
punctures back to boundary components. A point in the boundary circle
is a tangent direction out of the puncture. The foliation
$\hat{\mathcal F_i}$ lifts to a foliation $\mathcal F_i$ on $S_i$ with
$k$ leaves transverse to the boundary circle.  The homeomorphism $\hat
g_i$ naturally lifts to a homeomorphism $g_i:S_i\to S_i$, and it has
(at least) $k$ fixed points on the boundary circle, one for each prong
(in fact, it has at least $k$ more fixed points coming from the
transverse invariant foliation). In any case, there is a canonical way
to isotope $g_i$ to a homeomorphism that fixes the boundary pointwise,
keeping the $k$ points fixed throughout the isotopy.

We can now glue the surfaces $S_i$ together to form $S$, but we will
also insert an annulus between any two boundary components to be
glued. The purpose of this is that otherwise the homeomorphism of the
glued surface that agrees with $g_i$ on $S_i$ may not be $g$; it may
differ from $g$ by a product of Dehn twists in the curves in $\mathcal
C$. We realize any such Dehn twists on the inserted annuli. Extend
each $g_i:S_i\to S_i$ by the identity in the complement of $S_i$ to
obtain a homeomorphism of $S$, also denoted $g_i$. We summarize the
discussion as follows.

\begin{thm}
For every $g\in G$ there is $N>0$ such that
$$g=g_1\cdots g_m\delta_1^{n_1}\cdots\delta_p^{n_p}$$ where $\delta_j$
are (left) Dehn twists supported on annuli around the curves in the
reducing multicurve $\mathcal C$, $n_i\neq 0$ and each $g_i$ is a
pseudo-Anosov supported on a complementary subsurface $S_i$. Any two
homeomorphisms above commute.
\end{thm}

We now make some further definitions. First, after possibly taking a
further power of $g$, we may assume:
\begin{itemize}
\item each $g_i$ and $\delta_j^{n_j}$ is in $G$ (this is where we are
  using that $G$ has finite index in $MCG(S)$) and also in $\mathcal
  S$ (see Proposition \ref{BBF}),
\item for each $i$, either $g_i$ is conjugate (in $G$) to $g_i^{-1}$
  or $g_i^m$ is not conjugate to $g_i^{-m}$ for any $m>0$ (the latter
  is equivalent to saying that no $\gamma\in G$ interchanges the
  stable and unstable foliation of $g_i$).
\end{itemize}
If $g_i$ is conjugate to $g_i^{-1}$ in $G$ we say $g_i$ is {\it achiral}, and
otherwise it is {\it chiral}. 
If $g_i$ and $g_j$ are both chiral, we
say they are {\it equivalent} if a nontrivial power of $g_i$ is
conjugate in $G$ to a power of $g_j$. In other words, $g_i$ and $g_j$ are
equivalent if there is some element $\gamma\in G$ that takes $S_i$ to
$S_j$ and takes the stable foliation of $g_i$ to either the stable or
the unstable foliation of $g_j$.
We make the same definition for Dehn
twists (recall that a power of a Dehn twist is not conjugate to its
inverse, so we may view them as chiral): $\delta_i^{n_i}$ and
$\delta_j^{n_j}$ are {\it equivalent} if some of their nontrivial
powers are conjugate (equivalently, the corresponding annuli are in
the same $G$-orbit).

Let $\{g_{i_1},g_{i_2},\cdots,g_{i_p}\}$ be an equivalence class. Thus
$g_{i_1}^{m_1},g_{i_2}^{m_2},\cdots,g_{i_p}^{m_p}$ are all conjugate
for certain $m_j\neq 0$. We will say this equivalence class is {\it
  essential} if 
$$\frac 1{m_1}+\frac 1{m_2}+\cdots+\frac 1{m_p}\neq 0$$ and {\it
  inessential} otherwise (an example of an inessential class has appeared
in \cite{EnKo}).  Since $g_i^m$ conjugate to $g_i^n$ implies $m=\pm n$
for any $g_i$, and implies $m=n$ for chiral $g_i$, the exponents $m_i$
above are unique up to a common multiple. We make the same definition
for equivalence classes of powers of Dehn twists
$\{\delta_{i_1}^{n_1},\cdots,\delta_{i_p}^{n_p}\}$ with $m_i$ chosen
so that all $\delta_{i_j}^{n_jm_j}$ are pairwise conjugate.

\begin{thm}\label{chirality}
Let $G<MCG(\S)$ be a subgroup of finite index and $g\in G$. Then
$scl(g)>0$ if and only if some chiral equivalence class is essential.
\end{thm}

Note that if $G$ is a subgroup of $\mathcal S$ then every class is
chiral (by Corollary \ref{involution}) and has one element, so every
nontrivial $g\in G$ has $scl_G(g)>0$. 

\begin{proof}[Proof of Theorem \ref{chirality}]
We first prove that $scl(g)=0$ if every chiral class is
inessential. Let $H:G\to\R$ be a homogeneous quasi-morphism. We
will argue that $H(g)=0$. If $g_i$ is achiral then
$H(g_i)=H(g_i^{-1})$ so $H(g_i)=0$. Let
$\{g_{i_1},g_{i_2},\cdots,g_{i_p}\}$ be a chiral equivalence class
with $g_{i_1}^{m_1},\cdots,g_{i_p}^{m_p}$ all conjugate. Then
$H(g_{i_1}^{m_1})=\cdots=H(g_{i_p}^{m_p})$; call the common value
$A$. Thus $H(g_{i_j})=\frac A{m_j}$ and $H(g_{i_1}g_{i_2}\cdots
g_{i_p})=A(\frac 1{m_1}+\cdots+\frac 1{m_p})$, which is 0 for an
inessential class. A similar argument applies to inessential classes
of powers of Dehn twists. It now follows that $H(g)=0$ since $H$ is
additive on commuting elements.

Now assume that, after reindexing, $\{g_1,\cdots,g_p\}$ is an
essential chiral equivalence class. In case there are several such
classes we choose one with highest complexity ($=-\chi(S_1)$ where
$S_1$ is the surface supporting $g_1$). E.g. annuli have the smallest
complexity, so powers of Dehn twists would be chosen only if nothing
else is available. Further, in case there are several essential chiral
classes with maximal complexity we choose one whose primitive root has
largest translation length. That is, if $\{g_{i_1},\cdots,g_{i_s}\}$
is another essential chiral class of maximal complexity, and we write
$g_1=h_1^{N_1}$, $g_{i_1}=h_{i_1}^{N_{i_1}}$ with $N_1,N_{i_1}>0$
maximal possible and with $h_1$ [$h_{i_1}$] supported on the same
subsurface as $g_1$ [$g_{i_1}$], then $\tau_{h_1}\geq \tau_{h_{i_1}}$
(translation lengths are with respect to the curve graph of the
supporting subsurface).

We wish to construct a quasi-morphism $H:G\to\R$ which is unbounded on
the powers of $g_1 \cdots g_p$ but bounded on the powers of all other
chiral $g_j$ and Dehn twists that belong to essential classes. It will
then follow that $H$ is unbounded on the powers of $g$ and hence
$scl(g)>0$. For simplicity we assume $g_1$ is not a Dehn twist.

Let $G'=G\cap\mathcal S$, where $\mathcal S$ is the subgroup in
Proposition \ref{BBF}, so $G'$ is normal in $G$. We will now consider
the action of $G'$ on the graph $\cC(\bY)$ of Proposition \ref{bY},
where $\bY$ is the $\mathcal S$-orbit of $S_1$. According to
Proposition \ref{bY} $(G',\cC(\bY),g_1,C)$ satisfies $WWPD$
where $C<G'$ preserves the stable and unstable foliations of $g_1$ on
$S_1$. 

Choose coset representatives
$1=\gamma_1,\gamma_2,\cdots,\gamma_s\in G$ of $G/G'$.

Let $H':G'\to\R$ be the
associated quasi-morphism as in Corollary \ref{bf+++}. We are really
replacing $g_1$ here with a bounded power when applying this corollary
(see Remark \ref{bf++++}).
Finally, define $H:G'\to\R$ by
$$H(\gamma)=\sum_{i=1}^sH'(\gamma_i\gamma\gamma_i^{-1}).$$ It was
verified in \cite[Section 7]{rank1} that $H$ extends to a
quasi-morphism on $G$. (If we first replace $H$ by the homogeneous
quasi-morphism $\hat H$, then $\hat H(f)=\frac 1n\hat H(f^n)$ when
$f^n\in G'$ extends $\hat H$ to $G$.  Alternatively, we can first
replace $H'$ by $\hat H'$, then define $H$, which
is automatically homogeneous). Note that rechoosing the coset
representatives changes $H$ by a uniformly bounded amount.

\vskip 0.5cm 
{\bf Claim 1.} $H$ is unbounded on the powers of
$g_1$. 
\vskip 0.5cm
{\it Proof of Claim 1.}
The summand $H'(g_1^N)$ corresponding to the trivial
coset is unbounded on the powers of $g_1$ by construction (see Corollary
\ref{bf+++}) and we only
need to see that summands $H'(\gamma_ig_1^N\gamma_i^{-1})$ are bounded
or have the same sign as $H'(g_1^N)$. The support of
$\gamma_ig_1^N\gamma_i^{-1}$ is the surface $\gamma_i(S_1)$ where
$S_1$ is the support of $g_1$.

If $\gamma_i(S_1)\not\in\bY$ then by Lemma \ref{bounded.projection}
$\gamma_ig_1\gamma_i^{-1}$ has virtual quasi-axes that intersect every
$\cC(S)$ in a uniformly bounded segment, so their projections to the
translates of virtual quasi-axes $\ell$ of $g_1$ are uniformly bounded. It
now follows that $H'(\gamma_ig_1^N\gamma_i^{-1})$ is bounded by
Proposition \ref{bf+++}(d).

If $\gamma_i(S_1)\in\bY$ then $\gamma_ig_1\gamma_i^{-1}$ preserves
$\cC(\gamma_i(S_1))$. Its translation length on $\cC(\gamma_i(S_1))$
is $\tau_{g_1}$ since $\gamma_i$ conjugates the action of $g_1$ on
$\cC(S_1)$ and the action of $\gamma_ig_1\gamma_i^{-1}$ on
$\cC(\gamma_i(S_1))$.  Thus it follows from Lemma \ref{brian} that the
projection of a quasi-axis of $\gamma_ig_1\gamma_i^{-1}$ to a
$G'$-translate of $\ell$ is either bounded by a linear function of
$\tau_{g_1}$, or the two lines are parallel. In the former case, after
replacing $g_1$ by a definite power, we may assume the projections are
bounded by $\tau_{g_1}-R$ and so $H'$ is bounded on the powers of
$\gamma_ig_1\gamma_i^{-1}$ by Corollary \ref{bf+++}. In the latter case
$\gamma_ig_1\gamma_i^{-1}\in C$ and by chirality
$\gamma_ig_1\gamma_i^{-1}$ does not translate the opposite way from
$g_1$, so $H'(\gamma_ig_1^N\gamma_i^{-1})$ has the same sign as
$H'(g_1^N)$.

\vskip 0.5cm 
{\bf Claim 2.} $H$ is unbounded on the powers of
$g_1\cdots g_p$. 
\vskip 0.5cm
{\it Proof of Claim 2.} Denote by $\hat H$ the homogeneous
quasi-morphism bounded distance away from $H$. If $\hat
H(g_1^{m_1})=A\neq 0$ then $\hat H(g_i^{m_i})=A$, $\hat
H(g_i)=\frac A{m_i}$ and $\hat H(g_1\cdots g_p)=A(\frac
1{m_1}+\cdots+\frac 1{m_p})\neq 0$ since the class is essential.

\vskip 0.5cm 
{\bf Claim 3.} Let $\{g_{i_1},g_{i_2},\cdots,g_{i_q}\}$
be an equivalence class distinct from $\{g_1,\cdots,g_p\}$. Then $H$
is bounded on the powers of $g_{i_1}g_{i_2}\cdots g_{i_q}$.
 
\vskip 0.5cm {\it Proof of Claim 3.}  If the class is achiral or
inessential chiral then we showed that every quasi-morphism is bounded
on the powers. Now assume the class is essential chiral. The argument
is similar to Claim 1.  Consider a conjugate $\gamma_i
g_{i_1}\gamma_i^{-1}$. Let $S_{i_1}$ be the support of $g_{i_1}$. By
the maximality assumption the surface $\gamma_i(S_{i_1})$, which
supports $\gamma_i g_{i_1}\gamma_i^{-1}$, does not
properly contain any surface in $\bY$. If it is not equal to any
surface in $\bY$ then $H'$ is bounded on the powers of $\gamma_i
g_{i_1}\gamma_i^{-1}$ as in Claim 1. If $\gamma_i(S_{i_1})\in\bY$ then
we can apply Lemma \ref{brian} again since $\tau_{\gamma_i
g_{i_1}\gamma_i^{-1}}\leq \tau_{g_1}$ to deduce that $H'$ is bounded
on the powers of $\gamma_i
g_{i_1}\gamma_i^{-1}$
(it is not possible for the virtual quasi-axes to be
parallel here since $g_1$ and $g_{i_1}$ belong to distinct classes). It
now follows that $H$ is unbounded on the powers of $g$.

The argument when the the essential class consists of powers of Dehn
twists is similar. We then use the collection $\bY$ of annuli
consisting of the $G'$-orbit of the annuli supporting Dehn twists in
the collection to build a hyperbolic graph $X=\cC(\bY)$. The
role of the curve graph $\mathcal C(S_1)$ is played by the curve
graph of the annulus, which is quasi-isometric to $\R$.
\end{proof}

We now state a number of consequences of the above proof. Some of them
require going back and checking a few things.

\subsection{Separability}
We say that $g,g'\in G$ are {\it inseparable} if for any two
homogeneous quasi-morphisms $H,H':G\to\R$ the vectors
$(H(g),H'(g))\in\R^2$ and $(H(g'),H'(g'))\in\R^2$ are linearly
dependent. Otherwise, $g,g'$ are {\it separable} (see \cite{PoRu}, \cite{EnKo}). 

\begin{itemize}
\item 
If $g$ has
$scl(g)=0$ then $g,g'$ are inseparable for every $g'$. 
\item
$g^n$ and $g^m$ are inseparable for every $m,n$,
\item
if 
$g,h$ are (in)separable, then $g^n,h^m$ are (in)separable for any
$m,n\neq 0$.
\end{itemize}

Suppose $g,g'\in G$ have $scl(g)>0$, $scl(g')>0$. After passing to
powers we may assume each can be written as the product of powers of
Dehn twists and pseudo-Anosov homeomorphisms on subsurfaces, as
discussed above. Suppose there are essential chiral classes
$\{g_{i_1},\cdots,g_{i_p}\}$ for $g$ and
$\{g'_{j_1},\cdots,g'_{j_q}\}$ for $g'$ so that $g_{i_1}$ 
and $g'_{j_1}$ have conjugate powers (i.e. the two classes are
equivalent to each other). Then for any homogeneous quasi-morphism
$H:G\to \R$ the ratio
$$\frac{H(g_{i_1}\cdots g_{i_p})}{H(g'_{j_1}\cdots g'_{j_q})}$$
does not depend on $H$ (as long as it is not $\frac 00$). We call it
the {\it characteristic ratio} of the essential chiral class that
occurs in both $g$ and $g'$. We make a similar definition for
conjugacy classes of powers of Dehn twists that occur in both $g$ and
$g'$. Note that the characteristic ratio is always rational, and it
can be computed from knowing which powers of the mapping classes in it
are conjugate. 

\begin{prop}
Let $g,g'\in G$ be two elements with $scl(g)>0$, $scl(g')>0$. Then
$g,g'$ are inseparable if and only if every essential chiral class of
$g$ also occurs in $g'$ and vice-versa, and all characteristic ratios
are equal.
\end{prop}

\begin{proof}
If there is an essential chiral class that occurs in $g$ but not $g'$,
the proof of Theorem \ref{chirality} produces a homogeneous
quasi-morphism with $H(g)\neq 0$ and $H(g')=0$, so $g,g'$ are
separable. Otherwise, for each characteristic ratio $r$ there is a
homogeneous quasi-morphism $H$ with $\frac{H(g)}{H(g')}=r$.

If all essential chiral classes for $g$ occur in $g'$ and vice-versa, and all
characteristic ratios are equal to $r$, then $\frac{H(g)}{H(g')}=r$
(or $\frac 00$) for any homogeneous quasi-morphism $H$, so $g,g'$
are inseparable.
\end{proof}

For example, Dehn twists in curves in different $G$-orbits are
separable. 

There is a more general statement along the same lines. Denote by
$\mathcal X$ the real vector space whose basis consists of equivalence
classes (over $G$) of
pure mapping classes which are chiral. In each class $[\gamma]$ choose a
representative $\gamma$. Thus if $g\in MCG(\Sigma)$ then a definite
power $g^N$ decomposes as a product of pure classes, and after
ignoring achiral components we get an element $\chi(g)\in \mathcal X$ by
setting
$$\chi(g)=\sum_\gamma n_\gamma [\gamma]$$
where $n_\gamma$ is computed as follows. Say $g_1,\cdots,g_k$ are the
components of $g^N$ equivalent to $\gamma$, so that $g_i^{m_i}$ is
conjugate to $\gamma^{r_i}$. Then let $n_{\gamma}=\frac 1N\sum_i
\frac {r_i}{m_i}$. The arguments above show:

\begin{prop}\label{4.3}
Let $h_1,\cdots,h_p\in G$. The dimension of the space of
functions $\{h_1,\cdots,h_p\}\to\R$ which are restrictions of
homogeneous quasi-morphisms $G\to\R$ is equal to the dimension
of the subspace of $\mathcal X$ spanned by $\chi(h_1),\cdots,\chi(h_p)$.
\end{prop}

Even more generally, let $C_1(G)$ be the vector space of chains
$r_1h_1+\cdots+r_ph_p$ with $r_i\in\R$ and $h_i\in G$. Any
homogeneous quasi-morphism on $G$ extends by linearity to
$C_1(G)$ and there is a linear map $\chi:C_1(G)\to \mathcal X$
defined on the basis by the discussion above. 

\begin{prop}
Let $c_1,\cdots,c_p\in C_1(G)$. The dimension of the space of
functions $\{c_1,\cdots,c_p\}\to\R$ which are restrictions of
homogeneous quasi-morphisms $C_1(G)\to\R$ is equal to the dimension
of the subspace of $\mathcal X$ spanned by $\chi(c_1),\cdots,\chi(c_p)$.
\end{prop}

\subsection{Lower bound to $\scl$}

\begin{prop}\label{scl.bound}
Let $G$ be a finite index subgroup of $MCG(\S)$.
There is a constant $\epsilon=\epsilon(G,\S)>0$ so that the following holds.
Let $g\in G$ be any
element. If $scl_G(g)>0$ there is a homogeneous quasi-morphism
$H:G\to\R$ such that $\frac{H(g)}{2\Delta(H)}\geq\epsilon$.
\end{prop}

As a consequence, by 
Proposition \ref{p:scl}(iii)),
we have that $\scl_G(g)>0$ implies $\scl_G(g)\geq\epsilon$.

\begin{proof}
The argument is same as the one for Theorem \ref{chirality}, but 
we need to bound constants carefully. 
We only give an outline of the argument. In the special case $G=MCG(S)$
and $g\in G$ pseudo-Anosov class this statement was proved by
Calegari-Fujiwara in \cite{calegari-fujiwara}. 
As before, set $G'=G\cap\mathcal S$.

{\bf Step 1.}  In the course of the proof of Theorem \ref{chirality},
we replaced a given element $g$ by an arotational power $g^N$. We note
that the power can be taken to be uniformly bounded.  This follows
from the fact that the number of curves in a $g$-invariant pairwise
disjoint collection is uniformly bounded, the number of complementary
components is uniformly bounded, the number of singularities of the
stable foliation and the number of prongs is uniformly bounded at any
point. Moreover, after taking a further bounded power, the component
maps $g_i$ in the Nielsen-Thurston decomposition are in $G'$.  We will
also rename $g^N$ as $g$.

{\bf Step 2.} Since $scl(g)>0$, $g$ has an essential chiral class.
Among all essential chiral classes we choose one, say
$g_1,\cdots,g_p$, where the complexity (absolute value of the Euler
characteristic) of the supporting subsurface of each map in the class
is largest possible, and among these, we arrange that the primitive
root of $g_1$ has maximal translation length. Let $\hat{\mathcal F_1}$
be the stable foliation of $\hat g_1$ on the associated punctured
surface $\hat S_1$. We have a short exact sequence
$$1\to K\to Stab(\hat{\mathcal F_1})\to \Z\to 1$$
where the map to $\Z\cong \{\lambda^n\mid n\in\Z\}\subset \R^+$ is
given by the stretch factor, and $K$ is a finite group whose size is
bounded in terms of $\S$. We now use the following fact from group
theory: if two elements $\phi,\psi\in Stab(\hat{\mathcal F_1})$ have
the same image in $\Z$ then their powers $\phi^r,\psi^r$ are equal for
$r=|K|$. Let $\hat h$ denote an element of $Stab(\hat{\mathcal F_1})$ that
maps to $1\in\Z$. Then a certain bounded power $\hat h^N$ is arotational
and lifts to a homeomorphism $h$ on $S_1$ and extends by the identity to
$\Sigma$.
By taking a uniformly bounded power if necessary,
we also assume that $h \in G'$.
Putting all this together, we deduce that, perhaps after
replacing $g$ with a bounded power, for some integers
$n_1,\cdots,n_p$ each $g_i$ is conjugate to $h^{n_i}$. That the class
is essential implies that $n_1+\cdots+n_p\neq 0$.

{\bf Step 3.} We now construct $H':G'\to\R$ using the action
of $G'$ on
the space $X=\mathcal C(\bY)$, where $\bY$ is the $G'$-orbit of
the subsurface $S_1$ supporting $h$. To do this, we apply Corollary
\ref{bf+++} to a uniformly bounded power of $h$, thus:
\begin{itemize}
\item the defect $\Delta(H')$ is at most $12$,
\item $\hat H'(h)\geq\epsilon>0$.
\end{itemize}

Also note that $\Delta(\hat
H')\leq 4\Delta(H')\leq 48$ (from the definition $\hat
H'(\gamma)=\lim_{n\to\infty} H'(\gamma^n)/n$ we see that $|H'-\hat
H'|\leq \Delta(H')$ and the statement follows from the triangle inequality). 

{\bf Step 4.} Now let $\gamma_i\in G$ be the coset representatives of
$G/G'$ and define $\hat H'':G'\to\R$ by $\hat H''(\gamma)=\sum_i \hat
H'(\gamma_i\gamma\gamma_i^{-1})$. $\hat H''$ is homogeneous. Then
\begin{itemize}
\item $\Delta(\hat H'')$ is bounded above (by $[G:G']\Delta(\hat H')$,
  and $[G:G']\leq [MCG(S):\mathcal S]$),
\item $|\hat H''(h)|$ is bounded away from 0 (no two summands have
  opposite sign, by chirality). 
\end{itemize}
Then $\hat H''$ extends to a homogeneous quasi-morphism on $G$ (via
$\hat H(\gamma)=\hat H(\gamma^n)/n$ for $n=[G:G']$) and $\Delta(\hat
H)\leq [G:G']\Delta(\hat H'')$, see \cite[Lemma 7.2]{rank1}. 
It remains to show that $|\hat H(g)|$ is bounded away from 0. First,
$$\hat H(g_1\cdots g_p)=\hat H(g_1)+\cdots+\hat
H(g_p)=(n_1+\cdots+n_p)\hat H(h)$$ and since $|n_1+\cdots+n_p|\geq 1$
and $\hat H(h)=\hat H''(h)$ we see that $|\hat H(g_1\cdots g_p)|$ is
bounded away from 0. It now suffices to note that for $j>p$
$\hat H(g_j)=0$ whenever $g_j$ belongs to an essential
  chiral class
\end{proof}

\begin{remark}
Let $\S_h$ be the closed surface of genus $h$.  For $G=MCG(\S_g)$ it
would be interesting to know how $s_h:=\inf \{scl(g)\mid g\in G,
scl(g)>0\}$ behaves when the genus $h\to\infty$. On the plus side,
subsurface projection constants are uniform (see Leininger's proof in
\cite{mangahas,mangahas2}) and so is the hyperbolicity constant
$\delta$ of curve graphs, see \cite{bhb,hpw,crs}. The acylindicity
constants in Lemma \ref{brian} are known explicitly \cite{webb},
translation lengths in curve complexes are not uniform, but the
asymptotics is understood \cite{gadre-tsai}. The main deficiency of
our argument is that it passes to the subgroup $\mathcal S$, whose
index goes to $\infty$. There is a case where this can be avoided,
namely when the genus $h=2m$ is even and $g$ is the Dehn twist in a
curve that separates $\S_h$ into two subsurfaces of genus $m$. Then
all of $MCG(\S_{h})$ acts on $\mathcal C(\bY)$, where $\bY$ is the
$MCG(\S_{h})$-orbit of annuli containing the support of $g$. We
conclude that $scl(g)>\epsilon>0$ independently of $h=2m$. This
implies that $scl$ of a boundary Dehn twist is uniformly bounded below
(in fact it is $\frac 12$ by \cite{baykur}).
\end{remark}

\subsection{A bound on the commutator length}

\begin{prop}\label{cl.bounded}
For a finite index subgroup $G<MCG(S)$ and $g\in [G,G]$, $\scl(g)=0$
if and only if $cl(g^n)$ is bounded for $n\in\Z$. Moreover, there are
numbers $B=B(G)$ and $N=N(G)>0$ such that for every $g\in [G,G]$ with
$scl(g)=0$, $g^N$ can be written as a product of $B$ commutators.
\end{prop}

\begin{proof}
First note that the second statement implies the first, for if $g$ is
a product of $K=K(g)$ commutators and $g^{iN}$ is a product of $B$
commutators for every $i$, then every power of $g$ is a product of
$B+(N-1)K$ commutators.

If $g$ is achiral, i.e. if $\gamma g\gamma^{-1}=g^{-1}$ for some
$\gamma\in G$, then $\gamma g^k\gamma^{-1}=g^{-k}$, so
$g^{2k}=g^kh^{-1}g^{-k}h=[g^k,h^{-1}]$ is a single commutator.

If $g=g_1\cdots g_p$ is a single inessential chiral class, write
$g_i=h_i^{n_i}$ with all $h_i$ conjugate, say $h_i=\gamma_i h_{i-1}
\gamma_i^{-1}$ for $i=2,3,\cdots,p$. Then 
$$g=h_1^{n_1}h_2^{n_2}\cdots h_p^{n_p}$$
with $\sum n_i=0$. Thus we can write $g$ as a product of $(p-1)$
commutators:
$$g=[h_1^{n_1},\gamma_2][h_2^{n_1+n_2},\gamma_3]\cdots 
[h_{p-1}^{n_1+n_2+\cdots+n_{p-1}},\gamma_p]$$
\noindent
and note that $p$ is uniformly bounded by the topology of $\Sigma$.

Now $scl(g)=0$ implies, according to Theorem \ref{chirality}, that
some nontrivial power $g^k$ can be written as a commuting product of a
bounded number of achiral elements and inessential classes of chiral
elements. The claim follows from the observation that the power $k$ is
bounded in terms of $\Sigma$ and $G$ (see Step 1 in the proof of
Proposition \ref{scl.bound}).
\end{proof}

\subsection{Restrictions to subsurfaces}

Recall that for a group $G$ the vector space of all quasi-morphisms
modulo homomorphisms plus bounded functions is denoted by
$\widetilde{QH}(G)$ and this vector space is naturally isomorphic to
the kernel of the comparison map $H^2_b(G;\R)\to H^2(G;\R)$ from the
bounded to the regular cohomology of $G$. Also recall that when
$\Sigma$ supports pseudo-Anosov homeomorphisms the space
$\widetilde{QH}(MCG(\Sigma))$ is infinite dimensional \cite{bf:gt}.

\begin{prop}
Let $S\subset \S$ be a subsurface that supports pseudo-Anosov
homeomorphisms. Then the restriction map
$$\widetilde{QH}(MCG(\S))\to\widetilde{QH}(MCG(S))$$
has infinite dimensional image.
\end{prop}

\begin{proof}
Recall that \cite{bf:gt} produces a sequence $f_1,f_2,\cdots$ of
chiral, pairwise inequivalent pseudo-Anosov homeomorphisms on $S$. The
proof of the main theorem gives, for every $i$, a quasi-morphism
$H_i:MCG(\Sigma)\to\R$ which is unbounded on the powers of $f_i$, and
0 on all powers of $f_j$ for $j<i$. The statement also follows from
Proposition \ref{4.3}.
\end{proof}

\section{Example: Level subgroups and the Torelli group}\label{section.example}

We start by looking at the level subgroup $G=\Gamma_p$ for $p\geq 3$
consisting of mapping classes that act trivially in
$H_1(\Sigma;\Z_p)$. Recall the theorem of Ivanov \cite[Theorem 1.7]{ivanov} that
every $f\in G$ fixes the punctures, the Nielsen-Thurston reducing
curves are each invariant, and the restriction of $f$ to a
complementary component (after collapsing boundary to punctures) is
either identity or pseudo-Anosov. 
If at least one of them is pseudo-Anosov, we say $f$ has {\em exponential growth}.

Recall that two simple closed curves in $\Sigma$ are homologous over
$\Z_p$ if and only if either both are separating, or cobound a
subsurface (with compatible boundary orientation).

The following Lemma is left as an exercise (find a non-separating 
loop $b$ such that $f(b)$ and $b$ are not homologous if $i(a,f(a))=0$
and $f(a) \not=a$).

\begin{lemma}
If $f\in G$ and $a$ is a separating curve then either $f(a)=a$
  or $i(a,f(a))>0$, and in the former case $f$ preserves the
  orientation of $a$.\qed
\end{lemma}

\begin{lemma}
If $S$ is a subsurface of $\Sigma$ which is not an annulus and $f\in
G$ then $f(S)\cap S\neq\emptyset$.
\end{lemma}

\begin{proof}
Assume $f(S)\cap
  S =\emptyset$. 
First, by the previous lemma if the genus of $S$
  is $>0$ then $S$ contains a separating curve which cannot be moved
  off itself (and if fixed the orientation is preserved),
so $f(S)\cap
  S\neq\emptyset$, impossible. 
 Thus $S$ is
  a planar surface. If any of the boundary components are separating,
  they are fixed with the same orientation (otherwise, there will be a non-separating 
curve, in one of the components of the complement of the separating curve, whose homology class is not preserved by $f$), so $f(S)\cap
  S\neq\emptyset$, impossible. Thus they are all nonseparating, and in fact the
  only relation among them in homology is that the sum is 0 (this is equivalent to
  $\Sigma-S$ being connected), or
  otherwise $S$ contains a separating curve .
 Let $a$ be a boundary component, so
  $f(a)$ is a boundary component of $f(S)$. Since $a$ and $f(a)$
are homologous over $\Z_p$, there are two subsurfaces
  cobounded by $a$ and $f(a)$, one contains $S$ and the other contains
  $f(S)$. Denote by $A$ the one that contains $S$ (use $f(S)\cap
  S =\emptyset$). Similarly, choose
  another boundary component $b$ and let $B$ be the subsurface
  cobounded by $b$ and $f(b)$ that contains $S$. Then $A\cap B$ is a
  subsurface with two boundary components $a$ and $b$, showing that
  $a+b=0$ in homology, i.e. $S$ is an annulus.
\end{proof}

\begin{cor}\label{exponential}
Let $g=g_1\cdots g_m\delta_1^{n_1}\cdots \delta_p^{n_p}$ be the
decomposition of $g$ into pure parts as in Section \ref{scl}. Then
every $g_i$ is chiral and forms its own equivalence class.

In particular, every element of $G$ of exponential growth has positive
$scl_G$.
\end{cor}

\begin{proof}
Suppose $\gamma\in G$ conjugates $g_i$ to $g_i^{-1}$. Restricting to
the supporting surface and collapsing boundary to punctures gives a
mapping class $\hat \gamma$ of finite order\footnote{any mapping class
  that conjugates a pseudo-Anosov homeomorphism to its inverse flips
  the axis in Teichm\"uller space and must have a fixed point, so it
  has finite order} that conjugates $\hat g_i$
to $\hat g_i^{-1}$. This is impossible by Ivanov's theorem (Corollary
\ref{involution}). The fact that $g_i$ is not equivalent to $g_j$ for
$i\neq j$ follows from the lemma.
\end{proof}

\begin{lemma}\label{2 curves}
Suppose $a,b$ are distinct nonseparating homologous curves so
  that $a,b,f(a),f(b)$ have pairwise intersection number 0 and
  $f\in G$. Then $f(a)=a$ and $f(b)=b$. 
\end{lemma}

\begin{proof}
The four curves are cyclically ordered along $\S$ and if the conclusion
fails $f$ takes a cobounding subsurface off of itself.
(Notice that it does not happen that $f(a)=b$ and $f(b)=a$
by \cite[Theorem 1.7]{ivanov}.)
\end{proof}

\begin{cor}\label{multitwist}
Let $f\in G$ be a multitwist in a multicurve $M$. If $M$ contains a
separating curve then $scl(f)>0$ and otherwise $scl(f)>0$ if and only
if the sum of the powers of Dehn twists over some homology class of
curves in $M$ is nonzero.
\end{cor}

\begin{proof}
After taking a power, the twists in all curves occur with power
divisible by $p$. It does not happen that there are two distinct (up to
homotopy) separating curves $a,b$ and 
$f \in G$ with $f(a)=b$ since otherwise there will be a non-separating 
curve $c$ (in a component of the complement of $a$) whose homology 
class is not fixed by $f$.
Therefore if $M$ contains a separating curve then $scl(f)>0$.
Homologous nonseprating curves are in the same
$G$-orbit (in fact, the same Torelli-orbit), so the Dehn twists in
these curves form a chiral class.
\end{proof}

\subsection{The Torelli group $\mathcal T$}

Now let $\mathcal T$ be the group of mapping classes acting trivially
in the integral homology of $\Sigma$ and let $f\in\mathcal T$. 

\begin{thm}\label{torelli}
If $f\neq 1$ then $scl_{\mathcal T}(f)>0$.
\end{thm}

\begin{proof}
Since $\mathcal T<\Gamma_3$, if $f$ grows exponentially we have
$scl(f)>0$ (Corollary \ref{exponential}). So suppose $f$ is a multitwist supported on the multicurve
$M$. If $M$ contains a separating curve then $scl(f)>0$ (Corollary \ref{multitwist}), and otherwise
$f$ has infinite order in the abelianization of $\mathcal T$, see
\cite{BBM}, therefore $scl(f)>0$ by Corollary \ref{multitwist}.
\end{proof}
\bibliography{./ref2}

\end{document}